\documentclass{article}
\usepackage[utf8]{inputenc}
\usepackage{amsmath, amssymb, enumerate, multicol, chngcntr}
\usepackage{amsthm}
\usepackage{xcolor}
\usepackage{pdfpages}
\usepackage{comment}
\usepackage{caption}
\usepackage{bm}
\usepackage{mathtools}

\theoremstyle{plain}
\newtheorem{thm}{Theorem}[section]
\newtheorem{cor}[thm]{Corollary}
\newtheorem{prop}[thm]{Proposition}
\newtheorem{lemm}[thm]{Lemma}
\newtheorem{rem}{Remark}

\counterwithin{equation}{section}

\usepackage{fancyhdr}
\usepackage{setspace}

\usepackage[english]{babel}
\usepackage[euler]{textgreek}
\usepackage{stmaryrd}

\usepackage{tikz}
\usetikzlibrary{shapes,arrows}
\usepackage{tikz-cd}
\usepackage{amssymb}
\usepackage{graphicx}
\graphicspath{{Images/}}
\usepackage[font=small,labelfont=bf]{caption}
\usepackage{pgf,tikz,pgfplots}
\pgfplotsset{compat=1.15}
\usepackage{mathrsfs}
\usepackage{diagbox}
\usepackage{float}
\usepackage{xcolor}
\usepackage{ragged2e}
\usepackage{enumerate}
\usepackage{upgreek}
\usepackage{multicol}

\theoremstyle{remark}

\theoremstyle{definition}

\usepackage[hidelinks]{hyperref}
\usepackage{bookmark}
\usepackage{cleveref}

\hypersetup{
	colorlinks   = true, 
	urlcolor     = blue, 
	linkcolor    = blue, 
	citecolor   = red, 
	citecolor   = blue,
}

\theoremstyle{plain}

\newtheoremstyle{note}
{3pt}
{3pt}
{}
{}
{\itshape}
{:}
{.5em}
{}

\newtheoremstyle{citing}
{3pt}
{3pt}
{\itshape}
{}
{\bfseries}
{.}
{.5em}
{\thmnote{#3}}

\theoremstyle{citing}

\newtheoremstyle{break}
{9pt}
{9pt}
{\itshape}
{}
{\bfseries}
{.}
{\newline}
{}

\let\lvert=|\let\rvert=|

\title{On the Quartic-free $A$-groups}
\author{Prashun Kumar \footnote{Corresponding author, Dr. B. R. Ambedkar University Delhi, Delhi 110006, India; \ E-mail:  prashun07kumar@gmail.com.}}

\begin{document}
	\fontfamily{cmr}\selectfont
	
	\maketitle
	
	\bigskip
	\noindent
	{\small{\bf ABSTRACT:}} A finite group is said to be quartic-free if its order is not divisible by $p^4$ of any prime $p$. A finite group is called an $A$-group if all of its Sylow subgroups are abelian. Objective of this paper is to provide explicit structure of a quartic-free $A$-group. Further in the process of providing the explicit structure we also determine the derived length of a solvable quartic-free $A$-group.
	
	\medskip
	\noindent
	{\small{\bf Keywords}{:}}
	quartic-free groups, $A$-groups, general linear groups, solvable groups, non-solvable groups, nilpotent groups.
	
	\medskip
	\noindent
	{\small{\bf Mathematics Subject Classification-MSC2020}{:} }
	20E28, 20E34, 20E45, 20F99. 
	
	\baselineskip=\normalbaselineskip
    \section{Introduction}
    In 1893, Otto H$\ddot{\text{o}}$lder described groups of order $p^3$ and $p^4$. Soon after, he arrived at a formula for the number of groups of order $n$ using the structure of groups of order $n$ when $n$ is square-free, that is, the square of no prime divides $n$.  (see \cite{H1893}, \cite{H1895}). Result of H$\ddot{\text{o}}$lder, Burnside and Zassenhaus \cite[Theorem 10.1.10]{R1995} shows that every finite group with cyclic Sylow subgroups of order $n$ is metacyclic with odd-order derived subgroup $G' \cong \mathbb{Z}_m$ and cyclic quotient $G/G'$ of order $l = n/m$.
    
    \medskip
    The structure of cube-free groups was studied by Heiko Dietrich and Bettina Eick. They also presented an algorithm to construct cube-free groups up to isomorphism of a given order using ${\rm GAP}$ (see \cite{DE2005}). Later, S. Qiao and C. H. Li gave a more explicit structure of cube-free groups (see \cite{QL2011}). 
    
    \medskip
    Objective of this paper is to expand the work of S. Qiao and C. H. Li to the quartic-free $A$-groups and obtain the explicit structure of quartic-free $A$-groups. Further we would like to point out that though the paper is specifically oriented on quartic-free $A$-groups however the methodology developed in this paper may allow to obtain the structure of $A$-groups of higher degree, that is $A$-groups whose orders are not divisible by $p^n$ for every prime $p$ where $n\geq 5$.  
    
    \medskip
    Throughout the paper, $p$ is a prime, $q$ is a power of $p$ and $\mathbb{F}_q$ is the finite field of order $q$. Let $D(n,q)$ denote the subgroup of diagonal matrices of $\mbox{GL}(n,q)$. Let $M(n,q) = D(n,q) \rtimes S_n$ be the subgroup of monomial matrices in ${\rm GL}(n,q)$. Let $N(n,q)$ be  the normaliser of $S(n,q)$ where $S(n,q) \cong \mathbb{Z}_{q^n-1}$ is a Singer cycle in ${\rm GL}(n,q)$. The Borel subgroup $B(n,q)$ of ${\rm GL}(n,q)$ is defined as

   \begin{equation*}B(n,q) = 
   	 \left\{
   	\begin{pmatrix}
   		a_{11} & a_{12} & \cdots &  a_{1n} \\
   		  0    & a_{22} & \cdots &  a_{2n} \\
   		\vdots & \vdots & \ddots &  \vdots \\
   		  0    &   0    & \cdots &  a_{nn}
   	\end{pmatrix}
   	\mid a_{ii} \in {\mathbb{F}_q}^*, a_{ij} \in \mathbb{F}_q \right \}
   \end{equation*}
    
    \medskip
    Note that $B(2,q) \cong \mathbb{Z}_q \rtimes (\mathbb{Z}_{q-1} \times \mathbb{Z}_{q-1})$.
    
    \medskip
    Let $G$ be a group. We denote ${\rm sol}(G)$ to be the maximal solvable normal subgroup of $G$.

    \medskip
    Now we shall state the main results of our paper.
     
    \begin{thm}[{\bf Main Theorem 1}]\label{Main_theorem}
    	Let $G$ be a solvable quartic-free $A$ group. Then $G = A \rtimes ((B \rtimes (C \rtimes D)))$ where $A, B, C, D$ are abelian quartic-free subgroups $G$.
    	
    \end{thm}
    
    \begin{thm}[{\bf Main Theorem 2}]{\label{Main_theorem_2}}
    	Let $G$ be a non-solvable quartic-free $A$-group. Then $G = L \rtimes S$ where $L$ is a solvable and $S$ is a simple. Further if $S$ acts non trivially on $L$ then $5 \mid |S|$ and $S \cong A_5$.  
    \end{thm}

    \section{Solvable quartic-free $A$-groups}
    In this section we prove Theorem \ref{Main_theorem}. In the process of proving the result we classify solvable quartic-free $A$-subgroups of ${\rm GL}(2,q)$ and solvable quartic-free $p'$ $A$-subgroups of ${\rm GL}(3,q)$. Additionally we also provide the structure of solvable cube-free group of a given order. 
    
    \medskip
    The proof of Theorem \ref{Main_theorem} is divided into several results. We begin our proof by stating a standard result. 
     
    \begin{lemm}
    	Let $G$ be a solvable group and let $d(G)$ denote the derived length of $G$. Then
    	\begin{enumerate}[$(\rm i)$]
    		\item if $H \leq G$ then $d(H) \leq d(G)$.
    		\item if $N \trianglelefteq G$ then $ d(G) \leq d(G/N) + d(N)$
    		\item $d(H_1 \times \cdots \times H_k) = max_{1\leq i \leq k}\{d(H_i)\}$ where $H_i$'s are solvable. 
    	\end{enumerate}
    \end{lemm}
    
     Now we recall an important result by D. Taunt (see \cite{T1949}).
    
    \begin{lemm}{\label{Strctre_of_sol_A_grps}}
    	Let $G$ be a solvable $A$-group. Then $G = G' \rtimes N_1$ where $N_1$ is the system normalizer of $G$.
    \end{lemm}
    
    Since a subgroup of a solvable $A$-group is also a solvable $A$-group we can apply Lemma \ref{Strctre_of_sol_A_grps} to $G'$ and obtain that $G' = G'' \rtimes N_2$ where $N_2$ is the system normalizer of $G'$. So from Lemma \ref{Strctre_of_sol_A_grps} we have $G = (G'' \rtimes N_2) \rtimes N_1$. Continuing in this fashion we obtain the following corollary.
    
    \begin{cor}\label{cor_to_Strctre_of_sol_A_grps}
    	Let $G$ be a solvable $A$-group. Then $G = ((G^{m-1} \rtimes N_{m-1}) \rtimes \cdots )\rtimes N_1$ where $N_i$ is the system normalizer of $G^{i-1}$ and $m$ is the derived length of $G$.
    \end{cor}
    
    Since  a system normalizer of a solvable group is nilpotent therefore  a system normalizer of a solvable $A$-group is abelian. Thus by the Corollary \ref{cor_to_Strctre_of_sol_A_grps} a solvable $A$-group is largely a chain of semi-direct products of abelian groups.
    
    \medskip
    Now let $G$ be a solvable $A$-group and let $F$ be the Fitting subgroup of $G$. Since $F$ is nilpotent, $F = P_1 \times \cdots \times P_s$ where $P_i$ is the Sylow $p_i$-subgroup of $F$. Therefore 
    \begin{equation*}\label{eq_i}
    	G/F = N_G(F)/C_G(F) \leq {\rm Aut}(F) \cong \prod_{i=1}^{s}{\rm Aut}(P_i).\tag{$*$}
    \end{equation*}
    
    Let $K_i$ be the $i^{th}$ projection of $G/F$ into ${\rm Aut}(P_i)$. Then $G/F$ is embedded into $K_1 \times \cdots \times K_s$. Therefore by the properties of derived length we have 
    
    \begin{equation*}\label{eq_ii}
       d(G) \leq max_{1 \leq i \leq s}\{d(K_i)\} + 1.	\tag{$**$}
    \end{equation*}
    
    \medskip
    Therefore in order to obtain the derived length of a solvable quartic-free $A$-group we now examine the structure of subgroups of ${\rm Aut}(P)$ where $P$ is an abelian quartic-free $p$-group.
    
    \begin{prop}\label{subgrps_of_GL}
    	Let r be a prime. Let $H$ be a solvable irreducible $A$-subgroup of ${\rm GL}(r,q)$. Then one of the following holds.
    	\begin{enumerate}[$(\rm i)$]
    		\item $H$ is imprimitive and $H$ conjugates to a subgroups of $M(r,q)$.
    		\item $H$ is primitive and $H$ conjugates to a subgroup of $N(r,q)$.
    	\end{enumerate}
    \end{prop}
    
    \begin{proof}
    	If $H$ is an imprimitive subgroup of ${\rm GL}(r,q)$ then by \cite[Proposition 6.15, page 58]{BNV2007}, $H$ conjugates to a subgroup of $M(r,q)$.
    	
    	\medskip
    	So assume that $H$ is a primitive subgroup of ${\rm GL}(r,q)$. If $H$ is abelian then by \cite[Theorem 2.3.3, page 15]{S1992}, $H$ conjugates to a subgroup of $S(r,q) \leq N(r,q)$. So let $H$ be a non abelian subgroup of ${\rm GL}(r,q)$ and let $F$ the Fitting subgroup of $H$. Let $V$ be an irreducible $\mathbb{F}_qH$-module. Then by Clifford's Theorem $V$ is a semi-simple ${\mathbb{F}_qF}$-module of dimension $r$. Since $r$ is a prime either $V$ is a direct sum of isomorphic $1$-dimensional ${\mathbb{F}_qF}$-submodules or $V$ is an irreducible ${\mathbb{F}_qF}$-module. If $V$ is a direct sum of $1$-dimensional isomorphic ${\mathbb{F}_qF}$-submodules, then $F$ is a subgroup of ${\rm GL}(r,q)$ of scalar matrices but this contradicts that $C_H(F) = F$. Therefore $V$ must be an irreducible ${\mathbb{F}_qF}$-module. Thus by \cite[Theorem 2.3.3, page 15]{S1992}, $F$ is conjugate to a subgroup of $S(r,q)$. Since $F \trianglelefteq H$ by \cite[Theorem 2.3.5, page 15]{S1992}, $H$ is conjugate to a subgroup of $N(r,q)$.  
    \end{proof}
   
    Next we study the structure of quartic-free $A$-subgroups of ${\rm GL}(r,q)$ for $r \in \{2,3\}$. 
    
    \begin{lemm}{\label{cor_to_subgrps_of_GL}}
    	Let $H$ be a solvable quartic-free $A$-subgroup of ${\rm GL}(2,q)$. Then one of the following holds.
    	\begin{enumerate}[$(\rm i)$]
    		\item $H$ conjugates to a subgroup of $B(2,q)$. Further $H \cong P \rtimes (\mathbb{Z}_l \times \mathbb{Z}_s)$ where $P$ is the Sylow $p$-subgroup of $H$ and  $l \mid q-1$ and $s \mid q-1$.
    		\item $H$ conjugates to a subgroup of $D(2,q)$ and $H \cong \mathbb{Z}_l \times \mathbb{Z}_s$ where $l \mid q-1$ and $s \mid q-1$.
    		\item $H$ conjugates to a subgroup of $M(2,q)$ and $H = K \rtimes P$ where $K$ is Hall $2'$-subgroup of $H$ contained in $D(2,q)$ and $P$ is a Sylow $2$-subgroup of $H$.
    		\item $H$ conjugates to a subgroup of $N(2,q)$. $H = K \rtimes P$ where $K$ is Hall $2'$-subgroup of $H$ contained in $S(2,q)$ and $P$ is a Sylow $2$-subgroup of $H$.
    	\end{enumerate}
    \end{lemm}
    
    \begin{proof}
    	If $H$ is reducible but not completely reducible then $H$ conjugates to a subgroup of $B(2,q)\cong \mathbb{Z}_q \rtimes (\mathbb{Z}_{q-1} \times \mathbb{Z}_{q-1})$. In particular $H \cong P \rtimes (\mathbb{Z}_l \times \mathbb{Z}_s)$ where $P$ is the Sylow $p$-subgroup of $H$ and  $l \mid q-1$ and $s \mid q-1$.
    	
    	\medskip 
    	If $H$ is reducible and $p\nmid |H|$, then by Mashke's Theorem $H$ is completely reducible. Thus the underlying $\mathbb{F}_qH$-module is a direct sum of two one dimensional $\mathbb{F}_qH$-submodule. Therefore $H$ conjugate to a subgroup of $D(2,q) \cong \mathbb{Z}_{q-1} \times \mathbb{Z}_{q-1}$.
    	
    	\medskip
    	If $H$ is primitive then by Proposition \ref{subgrps_of_GL}, $H$ conjugates to a subgroup of $M(2,q)$. Let $H' = H \cap D(2,q)$ and $K$ be a Sylow $2'$-subgroup of $H'$. Then clearly $H = K \rtimes P$.
    	
    	\medskip
    	If $H$ is primitive then by Proposition \ref{subgrps_of_GL}, $H$ conjugates to a subgroup of $N(2,q)$. Let $H' = H \cap S(2,q)$ and $K$ be a Sylow $2'$-subgroup of $H'$. Clearly $H = K \rtimes P$.
    \end{proof}
    
    \begin{rem}
    	Note that an odd order $p'$ quartic-free $A$-subgroup of ${\rm GL}(2,q)$ is abelian.
    \end{rem}
    
    \begin{lemm}
    	Let $H$ be a $p'$ quartic-free $A$-subgroup of ${\rm GL}(3,q)$. Then one of the following holds.
    	\begin{enumerate}[$(\rm i)$]
    		\item $H$ is reducible and $H$ is conjugate to a subgroup of $D(3,q)$. In particular $H \cong \mathbb{Z}_l \times \mathbb{Z}_m \times \mathbb{Z}_s$ where $l \mid q-1$, $m \mid q-1$ and $s \mid q-1$. 
    		\item $H$ is reducible and $H$ is conjugate to a subgroup of ${\mathbb{F}_q}^* \times M(2,q)$. In particular $H = H_{2'} \rtimes H_2$ with $H_{2'}$ abelian. 
    		\item $H$ is reducible and $H$ is conjugate to a subgroup of ${\mathbb{F}_q}^* \times N(2,q)$. In particular $H = H_{2'} \rtimes H_2$ with $H_{2'}$ abelian.
    		\item $H$ is irreducible and $H$ is conjugate to a subgroup of $M(3,q)$ with $3 \mid |H|$. In particular if $|H|$ is odd then $H = H_{3'} \rtimes H_3$. If $|H|$ is even then  $H = H_{\{2,3\}'} \rtimes H_{\{2,3\}}$ where $ H_{\{2,3\}'} \leq D(3,q)$.
    		\item $H$ is reducible and $H$ is conjugate to a subgroup of $N(3,q)$. In particular either $H$ is cyclic or $H = H_{3'} \rtimes H_3$ with $H_{3'}$ cyclic.
    	\end{enumerate}
    \end{lemm}
    \begin{proof}
    Let $V$ be the underlying $\mathbb{F}_qH$-module. If $V$ is reducible and direct sum of three $1$-dimensional $\mathbb{F}_H$-submodules of $V$. Then $H$ conjugates into a subgroup of $D(3,q)$. 
    
    \medskip
    If $V$ is reducible and direct sum of $V_1$ and $V_2$ where $V_1$ is a $1$-dimensional $\mathbb{F}_qH$-submodule and $V_2$ is a $2$-dimensional irreducible $\mathbb{F}_qH$-submodule. Then by Lemma \ref{cor_to_subgrps_of_GL}, either $H$ conjugates to a subgroup of  ${\mathbb{F}_q}^* \times M(2,q)$ or to a subgroup of ${\mathbb{F}_q}^* \times N(2,q)$. Now first assume that $H \leq {\mathbb{F}_q}^* \times M(2,q)$. Let $L = {\mathbb{F}_q}^* \times D(2,q)$. Then $H \cap L$ is normal in $H$. Thus $H_{2'} \leq H \cap L$ is normal in $H$. Hence $H = H_{2'} \rtimes H_2$. Similarly we can show that $H = H_{2'} \rtimes H_2$ when $H$ is conjugate to a subgroup of ${\mathbb{F}_q}^* \times N(2,q)$.
    
    \medskip
    Now assume that $V$ is irreducible. If $V$ is imprimitive then by Proposition \ref{subgrps_of_GL}, $H$ is conjugate to a subgroup of $M(3,q)$. Further since $H$ is irreducible it permutes three $1$-dimensional  subspaces of $V$ transitively, therefore  $3 \mid |H|$. Now assume that $H \leq M(3,q)$. Let $\pi \in M(3,q)$ be the permutation matrix corresponding to a $3$-cycle. Then $H \leq D(3,q) \rtimes \langle \pi \rangle$. Therefore $H = H_{3'} \rtimes H_3$. Now let $K = H\cap D(2,q)$. Then clearly $H = K_{\{2,3\}'} \rtimes H_{\{2,3\}} =  H_{\{2,3\}'} \rtimes H_{\{2,3\}}$.
    
    \medskip
     So assume that $V$ is irreducible and primitive. Then by Proposition \ref{subgrps_of_GL}, $H$ is conjugate to a subgroup of $N(3,q)$. If $H$ is abelian then $H$ is cyclic. So suppose that $H$ is non abelian.  Let $K = S(3,q)\cap H$. Then $[H:K] = 3$. Thus $H = K_{3'} \rtimes H_3 = H_{3'} \rtimes H_3$.
    \end{proof}
    \begin{lemm}
    	Let $P \cong {\mathbb Z}_{p^2} \times {\mathbb{Z}_{p}}$. Then ${\rm Aut}(P) \cong R \rtimes ({\mathbb{Z}_{p-1}\times \mathbb{Z}_{p-1}})$ where $R$ is the Sylow $p$-subgroup of ${\rm Aut}(P)$ order $p^3$. In particular a $p'$-subgroup of ${\rm Aut}(P)$ is isomorphic to $\mathbb{Z}_l \times \mathbb{Z}_s$ where $l \mid p-1$ and $s \mid p-1$.
    \end{lemm}
    
    \begin{proof}
    Let $G = {\rm Aut}(P)$. A standard argument shows that $|G| = p^3(p-1)^2$. Thus by Sylow's Theorem it can be seen that $R$ is normal in $G$. So $G$ is a semi-direct product of $R$ by $G/R$. Now since $P_p = \{x \in P \mid x^p = 1\} \cong {\mathbb{Z}_p \times \mathbb{Z}_p}$ is characteristic in $P$, there is a homomorphism from $G/R$ to ${\rm GL}(2,p)$. Futher since the subgroup $P^p = \{x^p \mid x \in P\} \cong \mathbb{Z}_p$ is also characteristic in $P$ by Maschke's Theorem the image of $G/R$ conjugates into a subgroup of $D(2,p)$. Therefore $G/R \cong \mathbb{Z}_{p-1}\times \mathbb{Z}_{p-1}$.  	    
    \end{proof}
    
    Note that it is clear that an odd order quartic-free $A$-subgroup of ${\rm Aut}(P)$ where $P$ is an abelian group of order $p^{\alpha}$ with $\alpha \in \{1,2,3\}$ is meta abelian. Thus from Lemma \ref{Strctre_of_sol_A_grps} and equation (\ref{eq_ii}), we have the following structure of an odd order quartic-free $A$-group.
    
    \begin{cor}
    	Let $G$ be a quartic-free $A$-group of odd order. Then $G = (A \rtimes B) \rtimes C$ where $A,B$ and $C$ are abelian subgroups of $G$ of suitable orders.
    \end{cor}
    
    Before moving forward we provide the structure of a solvable cube-free group. Li and Qiao have already shown that a solvable cube-free group is largely a semi-direct product of cube-free abelian groups of suitable orders (see \cite{QL2011}). But our result provides description of these abelian groups. In particular we show that an odd order cube-free group $G$ is metaabelian with $G' \cong {\mathbb{Z}_a}\times \mathbb{Z}_b$ and $G/G' \cong {\mathbb{Z}_c} \times \mathbb{Z}_d$ where $a,b,c,d$ are suitable cube-free integers. Further in addition we prove that all the compliments of $G'$ in $G$ are conjugate.
    
    \begin{prop}
    	Let $G$ be a cube-free group of even order and let $H$ be a Hall $2'$-subgroup of $G$. Then 
    	\begin{enumerate}[$(\rm i)$]
    		\item $H = H' \rtimes N$ where $H'$ is abelian and $N$ is a system normaliser of $H$ containing a Sylow $3$-subgroup of $H$. Further all compliments of $H'$ in $H$ are conjugate.
    		\item  $G = H \rtimes P$ or $G = (P \times H') \rtimes N$ where $P$ is a Sylow $2$-subgroup of $G$.
    	\end{enumerate}
    \end{prop}
    
    \begin{proof}
    	Let $F$ be the Fitting subgroup of $H$. Let $\{p_1,\ldots,p_k\}$ be the set of all prime divisors of $|F|$ and let $F = P_1 \times \cdots \times P_k$ where $P_i$ is a Sylow $p_i$ subgroup of $F$. Then $H/F$ is embedded in $\prod_{i=1}^{k} {\rm Aut}(P_i)$. If is clear form Corollary \ref{cor_to_subgrps_of_GL} that $H/F$ is abelian and hence $H' \leq F$. Thus by Lemma \ref{Strctre_of_sol_A_grps}, $H = H' \rtimes N$ with $H'$ and $N$  abelian. Let $P$ be the Sylow $3$-subgroup of $H$. Let $r \neq 3$ be a prime dividing $|H|$ and let $R$ be a Sylow $r$-subgroup of $H$. Then $r \nmid |{\rm Aut}(P)|$. Thus by Burnside's compliment Theorem $P$ normalizes $R$ and hence $P$ conjugates to a subgroup of $N$. 
    	
    	\medskip
    	Now if $N'$ is another compliment of $H'$ in $H$ then it is not difficult to see that $N'$ is also a system normalizer of $H$. Therefore $N'$ is conjugate to $N$ in $H$. 
    	
    	\medskip
    	Part ${\rm (ii)}$ of this result follows from \cite[Lemma 3.8]{QL2011}.
    \end{proof}
    
    \begin{rem}
    	The structure of non-solvable cube-free groups will be discussed in the next section.
    \end{rem}
    
    Now we investigate the structure of quartic-free $A$-groups whose order divides only two primes. First we consider the case when $(p,r) \neq (2,3)$.  
    
    \begin{lemm}{\label{2-primes}}
    	Let $p$ and $r$ be distinct primes with $p < r$ and $(p,r) \neq (2,3)$. Let $G$ be a quartic-free $A$-group of order $p^\alpha r^\beta$. Let $P$ be a Sylow $p$-subgroup of $G$ and $R$ be a Sylow $r$-subgroup of $G$. Then either $G = P \rtimes R$ or $G = R \rtimes P$.
    \end{lemm}
   \begin{proof}
   	Let $N = N_G(R)$. Then by Sylow's theorem $[G: N] \equiv 1\bmod{r}$. If $[G:N] = 1$ then $R \trianglelefteq G$ and $G = R \rtimes P$.
   	If $[G: N] = p$ then $N \trianglelefteq G$ and therefore $G = R \rtimes P$. 
   	If $\alpha > 1$ and $[G: N] = p^2$ then $r \mid p^2 - 1$. Since $r > p$ and $(p,r) \neq (2,3)$ this case is not possible.
   	If $\alpha = 3$ and $[G : N] = p^3$ then $N = R$ and by Burnside's compliment theorem \cite[Theorem 10.1.8]{R1995} we have $G = P \rtimes R$. 
   \end{proof}  
   
     Now we deal with the case $(p,r) = (2,3)$. This case require case by case analysis of automorphism group of Sylow $2$-subgroup of $G$.
    
    \begin{lemm}
    	Let $G$ be a quartic-free $A$-group of order $2^\alpha 3^\beta$. Let $P$  and $R$ be Sylow $2$ and Sylow $3$-subgroups of $G$ respectively. Then one of the following holds.
    	\begin{enumerate}[$(\rm i)$]
    		\item $G = R \rtimes P$.
    		\item $G = P \rtimes Q$ where $P \in \{\mathbb{Z}_2 \times \mathbb{Z}_2, \mathbb{Z}_2 \times \mathbb{Z}_2\times \mathbb{Z}_2\}$.
    		\item  $P \cong  \mathbb{Z}_2 \times \mathbb{Z}_2$ and $G \cong \mathbb{Z}_9 \times (A_4 \times \mathbb{Z}_3)$.
    		\item $P \cong \mathbb{Z}_2 \times \mathbb{Z}_2\times \mathbb{Z}_2$ and $G \cong  (\mathbb{Z}_9 \rtimes \mathbb{Z}_2) \times (A_4 \times \mathbb{Z}_3)$.  
    		
    		 
    		\item  $P \cong  \mathbb{Z}_2 \times \mathbb{Z}_2$ and $G \cong (\mathbb{Z}_3 \times \mathbb{Z}_3) \times (A_4 \times \mathbb{Z}_3)$.
    		\item  $P \cong \mathbb{Z}_2 \times \mathbb{Z}_2\times \mathbb{Z}_2$ and $G \cong  (\mathbb{Z}_3 \times \mathbb{Z}_3) \rtimes \mathbb{Z}_2) \times (A_4 \times \mathbb{Z}_3)$. 
    		\item $G \cong (\mathbb{Z}_2 \times \mathbb{Z}_2) \rtimes (R \rtimes \mathbb{Z}_2)$ where $P \cong \mathbb{Z}_2 \times \mathbb{Z}_2\times \mathbb{Z}_2$.
    	\end{enumerate} 
    \end{lemm}
    \begin{proof}
    	If $P \not \in \{\mathbb{Z}_2 \times \mathbb{Z}_2, \mathbb{Z}_2 \times \mathbb{Z}_2\times \mathbb{Z}_2\}$ then ${\rm Aut}(P)$ is a $2$-group. Therefore $N_G(P) = C_G(P)$ and by Burnside's compliment Theorem $G = R \rtimes P$. So assume that $P$ is elementary abelian. Let $N = N_G(R)$. If $[G:N] \neq 4$ then as in the proof Lemma \ref{2-primes} either $G = R \rtimes P$ or $G = P \rtimes R$. Suppose that $[G: N] = 4$.
    	
    	\medskip
    	If $\beta = 1$ then $[G : N_G(P)] \in \{1,3\}$. If $[G: N_G(P)] = 1$ then $P \trianglelefteq G$ and $G = P \rtimes R$. If $[G : N_G(P)] = 3$ then $N_G(P) = C_G(P)$ and $G = R \rtimes P$.
    	
    	\medskip
    	If $\beta = 2$ and $[G:N_G(P)] \neq 3$ then the above arguments can be repeated to show that $G = P \rtimes R$ or $G = R \rtimes P$. Suppose  $[G: N_G(P)] = 3$ then there is a homomorphism $\phi:G \rightarrow S_3$ with $\ker(\phi) \leq N_G(P)$. If $8\mid |\ker(\phi)|$ then $P \leq \ker(\phi)$ and $G = P \rtimes R$. Thus $4 \mid \ker(\phi)$ and by Frattini's argument $G = P\cap \ker(\phi) \rtimes N_G(R) \cong (\mathbb{Z}_2 \times \mathbb{Z}_2) \rtimes (R \rtimes \mathbb{Z}_2)$.
    	
    	\medskip
    	If $\beta = 3$ and $[G: N_G(P)] \neq 9$ then the structure of $G$ can be obtained as in previous cases. So we assume that $[G: N_G(P)] = 9$ and $[G:N] = 4$. As $[G: N] = 4$ there is a homomorphism $\psi: G \rightarrow S_4$ with $\ker(\psi) \leq N$. If $R \leq \ker(\psi)$ then $G = R \rtimes P$. So assume that $R$ is not contained in $\ker(\psi)$. Then $K = P\ker(\phi)$ is a subgroup of $G$ of index $3$. Thus there is a homomorphism $\eta: G \rightarrow S_3$ such that $\ker(\eta) \leq K$. If $\ker(\eta) = K$ and $P \trianglelefteq K$. Then $G = P\rtimes R$. If $P$ is not normal in $K$. Then by Frattini's argument we have $G = KN_G(P) = \hat{R} \rtimes N_G(P)$ where $\hat{R}$ is a Sylow $3$-subgroup of $K$ of order $9$. Thus either $G \cong \mathbb{Z}_9 \rtimes (P \rtimes \mathbb{Z}_3)$ or $G \cong (\mathbb{Z}_3 \times \mathbb{Z}_3) \rtimes (P \rtimes \mathbb{Z}_3)$.
    	
    	\medskip
    	So assume that $\ker(\eta) \lneq K$. Then $M = R\ker(\eta)$ is a subgroup of $G$ of index $2$. Thus $M \trianglelefteq G$. If $R \trianglelefteq M$ then $G = R \rtimes P$. Otherwise by Frattini's argument $G = MN_G(R) \cong (\mathbb{Z}_2 \times \mathbb{Z}_2) \rtimes (R \rtimes \mathbb{Z}_2)$ this completes the proof.     
    \end{proof}
    
    \begin{lemm}
    	Let $G$ be a quartic-free $A$-group of even order. Let $P$ be a Sylow $2$-subgroup of $G$ and let $P \not \in \{{\mathbb Z}_2\times \mathbb{Z}_2, {\mathbb Z}_2\times \mathbb{Z}_2 \times \mathbb{Z}_2\}$. Then $G = H \rtimes P$ where $H$ is a Hall $2'$-subgroup of $G$.
    \end{lemm}
    \begin{proof}
    	Since $P$ is not elementary abelian ${\rm Aut}(P)$ is a $2$-group. Thus $P$ normalizes the Sylow system of a Hall $2'$-subgroup of $G$. Hence $G = H \rtimes P$.
    \end{proof}

\section{Non-solvable quartic-free A-groups}
In this section we prove Theorem \ref{Main_theorem_2}. We begin by classifying the simple quartic-free $A$-groups. Here we assume that the classification of finite simple groups holds.  
    
    \begin{lemm}
    	Let $G$ be a simple quartic-free $A$ group. Then $G \cong {\rm PSL}(2,q)$ for some $q = p^\alpha$ with $\alpha \in \{ 1,2,3\}$ and $q+1$ and $q-1$ quartic-free or $G \cong J_1$ first Janko group of order $2^3.3.5.7.11.19$.
    \end{lemm}

    \begin{proof}
    	By \cite{W1969} a simple non-abelian group with abelian Sylow-2 subgroup is isomorphic to $J_1$ or $L_2(q)$ for $q>3$ and $q \equiv 0,3$ or $5\bmod{8}$. But by \cite[page 191-214]{H1968}, $L_2(q) = {\rm PSL}(2,q)$.   
    \end{proof}
     
     Now we prove Theorem \ref{Main_theorem_2}. That is we now discuss the structure of a non-solvable quartic-free $A$-group.
     
     Recall that we denote ${\rm sol}(G)$ to be the largest normal solvable subgroup of a group $G$.
      
    \medskip
    \noindent
    {\bf Proof of Theorem 2.} Suppose that the statement of the Theorem is not true. Let $G$ be a minimal counter example and let $H = {\rm sol}(G)$. If $H ={1}$ then $G$ is semi-simple (simple in our case) and the result holds trivially. So assume that $H \not = {1}$ and let $N$ be the minimal normal subgroup of $G$ contained in $H$. Then $N$ is elementary abelian. By the minimal choice of $G$ we have $G/N = S/N \rtimes T/N$ where $S/N$ is solvable and $T/N$ is non-abelian simple. If $G \not = T$ then by minimal choice of $G$ we have $T = N \rtimes U$ where $U$ is non-abelian simple. Consequently $G = S \rtimes U$ contrary to the assumption. Thus $G = T$ and $G/N$ is simple. Therefore $G/N = (G/N)' = G'N/N$. If $G' \lneq G$ then by minimality $G' = R \rtimes K$ where $R$ is solvable and $K$ is simple. Consequently $G = (NR) \rtimes K$ contradicting that $G$ is a counter example. So we assume that $G = G'$. Since $G/N$ is semi-simple we have $C_G(N) = N$ or $C_G(N) = G$. If $C_G(N) = G$ then $Z(G) = N$ contradicting \cite[Theorem 10.1.7]{R1995}. Hence $C_G(N) = N$ but then ${\rm gcd}(|N|,[G:N]) = 1$ and by Schur-Zassenhaus \cite[Theorem 9.1.2]{R1995},  $G = N \rtimes M$ where $M$ is simple this contradiction  completes the proof.     


\begin{thebibliography}{50}
		\bibitem{BNV2007} S. R. Blackburn, P. M. Neumann, G. Venkataraman, \emph{Enumeration of finite Groups}, {Cambridge University Press}, 2007.
		\bibitem{DE2005} H. Dietrich, B. Eick, `On the group of cube-free order', \emph{J. Algebra} 292 (2005) 122-137.
		\bibitem{DL2021} H. Dietrich and D. Low, ``Generation of finite groups with cyclic Sylow Subgroups", \emph{Journal of Group Theory} {\bf 24} (2021) 161-175.
		\bibitem{H1968} B. Huppert, ``Endliche Gruppen. I", Springer, Berlin, 1968.
		\bibitem{H1893} Otto H$\ddot{\text{o}}$lder, `Die Gruppen der Ordnungen $p^3$, $pq^2$, $pqr$, $p^4$', \emph{Math. Annalen} (1893) 301-412.
		\bibitem{H1895} Otto H$\ddot{\text{o}}$lder, `Die Gruppen mit quadratfreier Ordnungszahl', \emph{Nachr. Gesellsch. Wiss. zu G$\ddot{\text{o}}$ttingen. Math.-phys. Klasse} (1895) 211--229.
		\bibitem{J1966} Z. Janko, ``A new finite simple group with abelian Sylow 2-subgroups and its characterization, {\emph{J. Algebra}} {\bf 3} (1966), 147-186.
		
		\bibitem{QL2011} S. Qiao and C. H. Li, `The finite groups of cube-free order', \emph{J. Algebra} {\bf 334} (2011) 101-108. 
		\bibitem{R1995} D. J. S. Robinson, \emph{A Course in the Theory of Groups}, Springer, New York, 1982 (Second Edition).
		\bibitem{S1992} M. W. Short, \emph{The Primitive Soluble Permutation Groups of Degree less than 256}, Springer-Verlag Heidelberg 1992.
		\bibitem{T1949} D. Taunt ``On A-groups", {\emph{Proc. Cambridge Philos. Soc.}} {\bf 45} (1949) 24-42.
		\bibitem{GV1997} Geetha Venkatraman, `Enumeration of finite soluble groups with abelian Sylow subgroups', \emph{Quart. J. Math. Oxford} {\bf 2} (1997) 107--125.
		\bibitem{W1969} J. H. Walter, ``The characterization of finite groups with abelian Sylow 2-subgroups, {\emph{Ann. of Math.}} (2) {\bf 89} (1969), 405-514.
		
	\end{thebibliography}
\end{document}